\documentclass[11pt]{article}
\usepackage{amsmath,amsthm}
\usepackage{amssymb,mathrsfs}
\usepackage{bm,mathtools}
\usepackage{extarrows}
\usepackage{tikz}
\usetikzlibrary{matrix,calc}
\tikzset{
  chart/.style={
  rectangle,
  rounded corners,
  draw=black, semithick,
  text width=6.2em,
  text centered}
}

\usepackage{xcolor}

\usepackage{geometry}
\usepackage[colorlinks=true,
linkcolor=blue,citecolor=blue,
urlcolor=blue,pagebackref]{hyperref}

\makeatletter
\def\@seccntDot{.}
\def\@seccntformat#1{\csname the#1\endcsname\@seccntDot\hskip 0.5em}
\renewcommand\section{\@startsection{section}{1}{\z@}%
{18\p@ \@plus 6\p@ \@minus 3\p@}%
{9\p@ \@plus 6\p@ \@minus 3\p@}%
{\large\bfseries\boldmath}}
\renewcommand\subsection{\@startsection{subsection}{2}{\z@}%
{15\p@ \@plus 6\p@ \@minus 3\p@}%
{6\p@ \@plus 6\p@ \@minus 3\p@}%
{\itshape}}
\renewcommand\subsubsection{\@startsection{subsubsection}{3}{\z@}%
{12\p@ \@plus 6\p@ \@minus 3\p@}%
{\p@}%
{}}
\makeatother

\usepackage{microtype}

\theoremstyle{plain}
\newtheorem{theorem}{Theorem}[section]
\newtheorem{lemma}{Lemma}[section]
\newtheorem{corollary}{Corollary}[section]

\newtheorem{problem}{Problem}[section]
\newtheorem{conjecture}{Conjecture}[section]

\theoremstyle{definition}
\newtheorem{definition}{Definition}[section]

\newtheorem{claim}{Claim}[section]

\numberwithin{equation}{section}
\allowdisplaybreaks
\DeclareMathOperator{\ex}{ex}

\DeclareMathOperator{\EX}{EX}

\title{On spectral Tur\'an theorems: confirming a conjecture of Guiduli and two problems of Nikiforov}

\author{Lele Liu\footnote{School of Mathematical Sciences, Anhui University, Hefei 230601,
P.R. China. E-mail: \texttt{liu@ahu.edu.cn} (L. Liu). Supported by the National
Nature Science Foundation of China (No. 12471320), and Anhui Provincial Natural Science Foundation for Excellent Young Scholars (No. 2408085Y003).}
~~ and ~~
Bo Ning\footnote{College of Cryptology and Cyber Science \& College of Computer Science, Nankai University, Tianjin 300350, P.R. China.
E-mail: \texttt{bo.ning@nankai.edu.cn} (B. Ning). Partially supported by the National Nature Science
Foundation of China (No. 12371350) and Fundamental Research Funds for the Central Universities, Nankai University (No. 63243151).}}

\date{}

\begin{document}
\maketitle

\begin{abstract}
Let $G$ be an $n$-vertex graph, and let $\lambda(G)$ and $\lambda_n(G)$ denote the largest and smallest eigenvalues of its adjacency matrix.  Write $e(G)$ for the number of edges of $G$, $d(G)=2e(G)/n$ for its average degree, and $T_r(n)$ for the $r$-partite Tur\'an graph on $n$ vertices.

We prove four sharp results in spectral Tur\'an theory. First, we confirm Guiduli's spectral dense-neighborhood conjecture (1996) in a stronger form: if $\lambda(G)\ge \lambda(T_r(n))$, then either $G\cong T_r(n)$, or there exists a vertex $v$ such that
$\lambda(G[N(v)]) > \lambda(T_{r-1}(d(v)))$.
Moreover, when $\lambda(G)>\lambda(T_r(n))$, every vertex attaining the maximum entry in any nonnegative Perron eigenvector of $G$ has this property. Second, we answer a problem of Nikiforov (2009) by showing that the exact Tur\'an edge threshold is detected by the exact spectral threshold: for every $r\ge 2$ and every $n$,
$\lambda(G)<\lambda(T_r(n))$, implying $e(G)<e(T_r(n)).$
Our proof also determines the equality cases.  Third, we answer another question of Nikiforov (2009) by showing that his least-eigenvalue clique bound
\[
\omega(G)\ge 1+\frac{2e(G)}{(n-d(G))(d(G)-\lambda_n(G))}
\]
does imply the concise form of Tur\'an's theorem. Finally, we discuss an open problem proposed by Ai et al. (2026) in \cite{ALNS26+}.

\par\vspace{2mm}
\end{abstract}


\section{Introduction}
\label{sec:1}

Extremal graph theory studies the maximum or minimum possible values of a graph parameter under prescribed structural restrictions.  Its classical starting point is Mantel's theorem for triangle-free graphs \cite{M07} and, more generally, Tur\'an's theorem \cite{T41}: if an $n$-vertex graph contains no copy of $K_{r+1}$, then it has at most $e(T_r(n))$ edges, where $T_r(n)$ is the complete $r$-partite graph whose parts are as equal as possible.  The Erd\H{o}s--Stone--Simonovits theorem \cite{ES46,Si68} shows that, asymptotically, the chromatic number of the forbidden graph governs the extremal number.  These theorems form one of the central lines of modern extremal graph theory.

Spectral extremal graph theory replaces the edge number by eigenvalues of matrices associated with the graph.  In this paper all eigenvalues are adjacency eigenvalues unless otherwise specified.
Given a graph $G$ on $n$ vertices, the \emph{adjacency matrix} $A(G)$ of $G$ is the $n$-by-$n$
matrix whose $(i,j)$-entry is equal to $1$ if the vertices $i$ and $j$ are adjacent and $0$ otherwise. Let $\lambda(G)$ be the spectral radius of $A(G)$. The Rayleigh quotient gives $\lambda(G)\geq 2e(G)/|V(G)|$; hence a sharp spectral upper bound usually implies the corresponding edge-extremal result. The converse direction is subtler: as emphasized in Nikiforov's survey \cite{N11}, spectral analogues often strengthen the classical theorem and sometimes have no purely edge-theoretic counterpart.  The Perron--Frobenius eigenvector is a central tool in this subject; for instance, Tait and Tobin \cite{TT17} used the leading eigenvector of a putative extremal graph to extract the structure of the extremal example.

The spectral Tur\'an theorem began with Nosal's theorem \cite{N70}, which states that if $G$ is triangle-free with $m$ edges, then $\lambda(G)\leq \sqrt m$, with equality for complete bipartite graphs. In particular, every triangle-free graph of order $n$ satisfies $\lambda(G)\leq \sqrt{\lfloor n^2/4\rfloor}$, with equality for $T_2(n)$. Wilf \cite{W86} proved the general clique-number bound $\lambda(G)\le (1-1/r)n$ for $K_{r+1}$-free graphs. The exact adjacency-spectral analog of Tur\'an's theorem was proved by Guiduli \cite{G96} and independently by Nikiforov \cite{N07}.

\begin{theorem}[Spectral Tur\'{a}n theorem \cite{G96,N07}]\label{Thm:spectral-turan}
Let $r\ge 2$, and let $G$ be an $n$-vertex graph containing no copy of $K_{r+1}$.  Then
$\lambda(G)\le \lambda(T_r(n)),$
and equality holds if and only if $G=T_r(n)$.
\end{theorem}

Theorem~\ref{Thm:spectral-turan} is a cornerstone of spectral extremal graph theory.  Guiduli's proof is close in spirit to degree-majorization and symmetrization arguments from the classical Tur\'an theory, whereas Nikiforov's proof is based on inequalities connecting the spectral radius with clique counts.  These two viewpoints have continued to shape the subject; see, for example, the spectral symmetrization refinements in \cite{LP23} and the clique-counting inequalities of Bollob\'as and Nikiforov \cite{BN07}.

For $K_{r+1}$-free graphs, Nikiforov \cite{N02} sharpened Theorem \ref{Thm:spectral-turan} to
$\lambda(G)^2 \leq 2e(G) ( 1- 1/r)$,
and obtained spectral versions of the Erd\H{o}s--Stone--Bollob\'as theorem \cite{N09-2}.  Since then, spectral Tur\'an problems have developed in several directions, including local and Perron-vector versions of the spectral radius \cite{LN26,LN-arxiv25,KKP-arxiv25,LSWW-arxiv26}, spectral supersaturation and subgraph counting \cite{BN07,NZ23,NZ25,LLZ26}, and exact spectral extremal results for specific forbidden subgraphs \cite{CFTZ20,WKX23,CDT24}.
In particular, Wang, Kang, and Xue \cite{WKX23} proved that for a graph $F$ satisfying $\ex (n,F) = e(T_r(n))+O(1)$ and for sufficiently large $n$, if $G$ has the maximum spectral radius among all $n$-vertex $F$-free graphs, then $G\in \EX(n,F)$. This is the first result that connects spectral extremal graphs to edge extremal graphs under the condition on the Tur\'an number of graphs.
A recent theorem of Byrne, Desai, and Tait \cite{BDT26+} gives a broad mechanism relating ordinary Tur\'an extremal graphs to spectral Tur\'an extremal graphs.  For further background and open problems, we refer to surveys \cite{N11,LFL22} and problem collection \cite{LN23}.

Our results address three questions at the interface of the classical and spectral theories: a conjecture of Guiduli on dense spectral neighborhoods, and two problems of Nikiforov concerning the extent to which spectral inequalities recover the corresponding Tur\'an statements.
We also give a partial result on an open problem in \cite{ALNS26+}.

\subsection{Guiduli's conjecture}
A classical strengthening of Tur\'an's theorem asks not only whether a graph has too many edges, but whether some neighborhood is already too dense. Erd\H{o}s \cite{E75} conjectured that if $e(G) > e(T_r(n))$, then some vertex $v$ has a neighborhood containing more than $e(T_{r-1}(d(v)))$ edges. He described this, if true, as a ``nice generalization of Tur\'an's theorem.'' The conjecture was proved independently by Bollob\'as and Thomason \cite{BT81} and by Erd\H{o}s and S\'os \cite{ES83}; Bondy \cite{B83} later showed that $v$ can be chosen as a maximum-degree vertex.

\begin{theorem}[\cite{BT81,ES83}]\label{thm:local-turan}
Let $G$ be a graph on $n$ vertices and at least $e(T_r(n))$ edges, where $r\geq 2$.
Then either $G\cong T_r(n)$, or there is a vertex $v$ in $G$ such that
$G[N(v)]$ has more than $e(T_{r-1}(d(v)))$ edges.
\end{theorem}

Guiduli \cite{G96} proposed the following spectral analog of this dense-neighborhood theorem.

\begin{conjecture}[Guiduli's spectral dense-neighborhood conjecture \cite{G96}]
Let $G$ be a graph on $n$ vertices and assume that $\lambda(G)\geq \lambda(T_r(n))$.
Then either $G\cong T_r(n)$, or there is a vertex $v$ such that $\lambda(G[N(v)])\geq \lambda(T_{r-1} (d(v)))$.
Furthermore, if $\lambda(G) > \lambda(T_r(n))$, then the conclusion
holds for any vertex $v$ having maximum weight given by a positive
eigenvector for $\lambda(G)$.
\end{conjecture}

Our first main theorem confirms Guiduli's conjecture in a stronger form.

\begin{theorem}\label{Thm:main1}
Let $G$ be a graph on $n$ vertices and assume that $\lambda (G)\geq \lambda (T_r(n))$.
Then either $G\cong T_r(n)$, or there is a vertex $v$ such that $\lambda (G[N(v)]) > \lambda (T_{r-1}(d(v)))$.
Furthermore, if $\lambda (G) > \lambda (T_r(n))$, then the conclusion
holds for any vertex $v$ attaining the maximum entry in any nonnegative Perron eigenvector of $\lambda (G)$.
\end{theorem}

Theorem~\ref{Thm:main1} strengthens the spectral Tur\'{a}n theorem in the same way that Theorem~\ref{thm:local-turan} strengthens Tur\'{a}n's theorem. Indeed, if $G$ is $K_{r+1}$-free, then every neighborhood $G[N(v)]$ is $K_r$-free.  Applying Theorem~\ref{Thm:spectral-turan} within each neighborhood gives
$\lambda(G[N(v)])\leq \lambda(T_{r-1}(d_G(v)))$
for every $v\in V(G)$.
Thus, Theorem~\ref{Thm:main1} rules out $\lambda(G)>\lambda(T_r(n))$ and recovers Theorem~\ref{Thm:spectral-turan}.

We shall prove the following stronger statement, which is the main technical step behind Theorem~\ref{Thm:main1}.

\begin{theorem}\label{Thm:join-preserve-spectral-inequality}
Let $r\geq 1$, $s\geq 1$ be integers, and
let $G$ be a graph on $n$ vertices.
The following conclusions hold:
\begin{enumerate}
\item[$(1)$] If
$\lambda(\overline{K}_s\vee G)\geq \lambda(\overline{K}_s\vee T_r(n))$,
then
$\lambda(G)\geq \lambda(T_r(n))$.
Moreover, $\lambda(G) = \lambda(T_r(n))$, then $G\in\mathcal{F}_r(n)$ and $\lambda(\overline{K}_s\vee G) = \lambda(\overline{K}_s\vee T_r(n))$.

\item[$(2)$] If
$\lambda(\overline{K}_s\vee G) > \lambda(\overline{K}_s\vee T_r(n))$,
then
$\lambda(G) > \lambda(T_r(n))$.

\item[$(3)$] If $\lambda(\overline{K}_s\vee G) = \lambda(\overline{K}_s\vee T_r(n))$ and $\lambda(G) = \lambda(T_r(n))$, then $G\in\mathcal{F}_r(n)$.
\end{enumerate}
\end{theorem}

\subsection{An open problem of Nikiforov}
Theorem~\ref{Thm:spectral-turan}, together with $\lambda(G)\ge 2e(G)/n$, immediately implies the usual Tur\'an theorem.  It is therefore natural to ask how much information is lost when one passes from the edge count to the spectral radius. Sometimes the implication can also be reversed in a useful way: Stanley's upper bound for the spectral radius in terms of the number of edges \cite{Stanley87} can be recovered from Tur\'an's theorem, a point made explicit by Nikiforov \cite{N20} after similar work of Hansen and Lucas \cite{HL10}.  In this spirit, Nikiforov \cite{N09} asked whether the exact Tur\'an edge threshold is detected by the exact spectral Tur\'an threshold.

\begin{problem}[Nikiforov \cite{N09}]\label{Prob:Main1}
Suppose that $G$ is a graph on $n$ vertices. Is it true that if $\lambda(G)<\lambda(T_r(n))$, then $e(G)<e(T_r(n))$?
\end{problem}

We answer Problem~\ref{Prob:Main1} positively in a stronger form.

\begin{theorem}\label{Thm:Prob1}
Let $G$ be a graph on $n$ vertices. If $e(G)\geq e(T_r(n))$,
then $\lambda(G)\geq \lambda(T_r(n))$. Moreover, if $e(G)\geq e(T_r(n))$ and
$\lambda(G) = \lambda(T_r(n))$,
then $G\in\mathcal{F}_r(n)$. Conversely, every $G\in\mathcal{F}_r(n)$ satisfies $e(G)=e(T_r(n))$ and $\lambda(G) = \lambda(T_r(n))$.
\end{theorem}

Equivalently, we prove the contrapositive: every $n$-vertex graph with at least $e(T_r(n))$ edges has spectral radius at least $\lambda(T_r(n))$.  The proof also determines all equality cases, which form an explicit family $\mathcal F_r(n)$ introduced in Section~\ref{Sec:2}.

The following corollary is Proposition 4.20 in \cite{G96}.
\begin{corollary}\label{coro:Prob1}
Let $G$ be a graph on $n$ vertices. If $e(G) > e(T_r(n))$,
then $\lambda(G) > \lambda(T_r(n))$.
\end{corollary}

\begin{proof}
Assume $e(G) > e(T_r(n))$. By Theorem \ref{Thm:Prob1}, we have $\lambda(G)\geq\lambda(T_r(n))$. If $\lambda(G)>\lambda(T_r(n))$, then the desired conclusion follows immediately. Hence, we may assume for contradiction that $e(G) > e(T_r(n))$ and $\lambda(G) = \lambda(T_r(n))$. By Theorem \ref{Thm:Prob1}, we obtain $G\in\mathcal{F}_r(n)$. It follows that $e(G)=e(T_r(n))$, a contradiction with $e(G) > e(T_r(n))$.
\end{proof}

\subsection{Spectral bounds for clique numbers and independence numbers}
For a graph on $n$ vertices, $\omega(G)$ denotes its clique number and $\lambda_n(G)$ denotes the least eigenvalue of its adjacency matrix.  Bounds for clique and independence numbers in terms of eigenvalues go back at least to Hoffman's ratio bound \cite{H70}; they were developed further through Wilf's bounds \cite{W86}, Haemers' interlacing method \cite{Haemers95}, work of Alon and Sudakov on the smallest eigenvalue and bipartite subgraphs \cite{AS00}, and Laplacian variants such as the theorem of Lu, Liu, and Tian \cite{LLT}.  Nikiforov's next theorem is a sharp nonregular adjacency-eigenvalue bound of this type.

\begin{theorem}[Nikiforov \cite{N09}]
 Let $G$ be a graph on $n$ vertices and $m$ edges, and let $d(G)=2m/n$. Then
 \begin{equation}
\label{eq:least-eigenvalue-bound}
  \omega(G)
  \ge
  1+
  \frac{2m}{\bigl(n-d(G)\bigr)
  \bigl(d(G)-\lambda_n(G)\bigr)}.
\end{equation}
\end{theorem}

Nikiforov asked whether \eqref{eq:least-eigenvalue-bound} is already strong enough to yield the concise form of Tur\'an's theorem.

\begin{problem}[Nikiforov \cite{N09}]\label{Prob:Main2}
 Is it possible to deduce the concise Tur\'an theorem from \eqref{eq:least-eigenvalue-bound}?
\end{problem}

We prove that the answer is yes. More precisely, \eqref{eq:least-eigenvalue-bound} implies that every $n$-vertex graph with average degree $d$ satisfies
\[
  \omega(G)\ge \frac{n}{n-d}
  =\frac{n^2}{n^2-2e(G)}.
\]
Consequently, if $e(G) > (1-1/r) n^2/2$,
then $G$ contains a clique of order $r+1$. This is the concise form of Tur\'an's theorem.

\begin{theorem}\label{Thm:Prob2}
The answer to Problem \ref{Prob:Main2} is positive.
\end{theorem}

\subsection{Notation and organization}

For a graph $G$ and a subset $X\subseteq V(G)$, let $G[X]$ denote the
subgraph of $G$ induced by $X$, and let $e_G(X)$ denote the number of edges in $G[X]$.
We also write $e(G)$ for the number of edges of $G$, and
$e_G(X, Y)$ for the number of edges with one vertex in $X$ and the other in $Y$.
As usual, for a vertex $v$
of $G$ we denote by $d_G(v)$ and $N_G(v)$ the degree of $v$ and the set of neighbors
of $v$ in $G$, respectively. In the above
notation, we will skip the subscript $G$ when $G$ is clear from context.
Given a fixed graph $F$, we say that $G$ is $F$-free if it contains no copy
 of $F$. The set of $F$-free graphs on $n$ vertices with the maximum number of edges is denoted by $\EX (n,F)$.
For vertex-disjoint graphs $G$ and $H$, we use $G \vee H$ to denote their \emph{join},
which is obtained by adding all possible edges between $G$ and $H$. For notation
and graph terminology undefined here, we refer the reader to \cite{BM08}.
Finally, $\bm{1}_n$ denotes the all-ones vector of length $n$.

The rest of the paper is organized as follows. Section \ref{Sec:2} presents the proof of Theorem \ref{Thm:Prob1} and the accompanying equality characterization. The next section establishes Theorems \ref{Thm:main1} and \ref{Thm:join-preserve-spectral-inequality}. Section \ref{Sec:4} proves Theorem \ref{Thm:Prob2}. The final section records two further problems and some comments on possible extensions.

\section{A complete solution to Problem \ref{Prob:Main1}}
\label{Sec:2}

This section is devoted to giving a positive answer to Problem \ref{Prob:Main1} in a stronger form. To this end, we introduce the following notation.

\begin{definition}[The family $\mathcal{F}_r(n)$]\label{def:family-Fn}
Let $r\geq 2$ and $n$ be integers. Write $n = ra+b$, $a\geq 1$, $0\leq b<r$. Define $\mathcal{F}_r(n)$ as follows.
\begin{enumerate}
\item[(1)] If $b=0$, then $\mathcal{F}_r(n)$ is the family of all
$(r-1)a$-regular graphs on $n$ vertices.

\item[(2)] If $1\leq b<r$, then $\mathcal{F}_r(n)$ is the family of all graphs $G$ having a partition
$V(G)=X\sqcup Y$, $|X|=b(a+1)$, $|Y|=(r-b)a$,
such that $X$ is completely joined to $Y$, and $G[X]$ is $(b-1)(a+1)$-regular, $G[Y]$ is $(r-b-1)a$-regular.
\end{enumerate}
\end{definition}

\begin{lemma}[{\cite[p.\,74]{CDS80}}]\label{lem:complete-r-partite}
Let $K_{n_1,n_2,\ldots,n_r}$ be the complete $r$-partite graph whose $r$ color classes have sizes $n_1,n_2,\ldots,n_r$, respectively. Then
the spectral radius $\lambda$ of $K_{n_1,n_2,\ldots,n_r}$ satisfies
$\sum_{i=1}^r \frac{n_i}{\lambda + n_i} = 1.$
\end{lemma}

We now prove Theorem~\ref{Thm:Prob1}.

\begin{proof}[Proof of Theorem \ref{Thm:Prob1}]
Put $p=b(a+1)$, $q=(r-b)a$, so that $p+q=n$.
The complement $\overline{T_r(n)}$ is the disjoint union of $b$ copies of $K_{a+1}$ and $r-b$ copies of $K_a$. Hence the number of non-edges of $T_r(n)$ equals
\begin{equation}\label{eq:m0}
  m_0:=e\bigl(\overline{T_r(n)}\bigr)
  =b\binom{a+1}{2}+(r-b)\binom{a}{2}
  =\frac{pa+q(a-1)}{2}.
\end{equation}
Consequently, $e(T_r(n))=\binom{n}{2}-m_0$.
Set $\lambda_0 = \lambda(T_r(n))$. By Lemma \ref{lem:complete-r-partite}, we have
\begin{equation}\label{eq:equ-lambda0}
\frac{p}{\lambda_0+a+1} + \frac{q}{\lambda_0+a} = 1.
\end{equation}

We shall also use the following elementary convexity fact.

\begin{claim}\label{claim:convexity}
Let $c_1,c_2,\ldots,c_n$ be integers with $0\leq c_i\leq n-1$ and
$\sum_{i=1}^n c_i \leq n(a-1)+p$. Then
\[
\sum_{i=1}^n \frac{1}{\lambda_0+c_i+1}
\geq \frac{p}{\lambda_0+a+1}+\frac{q}{\lambda_0+a}=1.
\]
Equality holds if and only if the multiset $\{c_1,c_2,\ldots,c_n\}$ consists of exactly $p$ copies of $a$ and $q$ copies of $a-1$.
\end{claim}

\begin{proof}[Proof of Claim \ref{claim:convexity}]
Define $f(x) = (\lambda_0+x+1)^{-1}$ for $x\geq 0$.
This function is strictly decreasing.
If $\sum_i c_i<n(a-1)+p$, then, since
$n(a-1)+p=pa+q(a-1)=2m_0\leq n(n-1)$
and each $c_i\leq n-1$, we can increase some $c_i$ by $1$, which strictly decreases $\sum_i f(c_i)$. Hence,
under the constraint $\sum_{i=1}^n c_i \leq n(a-1)+p$,
the sum $\sum_i f(c_i)$ is minimized only when $\sum_i c_i = n(a-1)+p$.

Now assume $\sum_i c_i=n(a-1)+p$. To minimize $\sum_i f(c_i)$,
the integers $c_i$ should be as equal as possible. Indeed, suppose two entries are $s$ and $t$ with $s\geq t+2$. Replacing them by $s-1$ and $t+1$ preserves the total sum and strictly decreases $f(s)+f(t)$, because
\begin{align*}
&f(s)+f(t)-f(s-1)-f(t+1)=\bigl(f(t) - f(t+1)\bigr) - \bigl(f(s-1) - f(s)\bigr) \\
= & ~\frac{1}{(\lambda_0+t+1)(\lambda_0+t+2)}
-\frac{1}{(\lambda_0+s)(\lambda_0+s+1)}>0.
\end{align*}
Iterating this step shows that a minimizing tuple must have all entries differing by at most $1$. Since the total sum is $\sum_{i} c_i = n(a-1)+p$,
the average value is $(n(a-1)+p)/n = a-1+p/n$.
Since $0\leq p<n$, the entries $c_i$ can only take the values
$a-1$ and $a$. Denote by $x$ the number of entries equal to
$a$. Then the remaining $n-x$ entries equal $a-1$, giving
$xa+(n-x)(a-1) = n(a-1)+x$. Comparing with $\sum_i c_i = n(a-1)+p$ yields $x=p$. Hence,
the unique such multiset consists of $p$ copies of $a$ and $n-p$ copies of $a-1$. This establishes both the inequality and the equality condition.
\end{proof}

We now continue the proof. Let $G$ be an $n$-vertex graph with $e(G)\geq e(T_r(n))$, and set
$H=\overline{G}$.
For each vertex $v$, write $c_v=d_H(v)$.
By \eqref{eq:m0} and the identity $e(T_r(n)) = \binom{n}{2} - m_0$, we obtain
$\sum_v c_v = 2e(H)\leq 2m_0 = n(a-1) + p$.
Define $x_v = f(c_v)$ and let $S=\sum_v x_v$.
By Claim \ref{claim:convexity} and \eqref{eq:equ-lambda0}, we see
$S\geq 1$. Since $A(G) = J-I-A(H)$, where $J$ is the all-ones matrix, it follows that
\[
\bm{x}^{\mathrm{T}} A(G) \bm{x} = S^2 - \sum_v x_v^2 - 2\sum_{uv\in E(H)} x_u x_v,
\]
where $\bm{x} = (x_v)_{v\in V(G)}\in\mathbb{R}^n$. On the other hand,
\[
\sum_{uv\in E(H)}(x_u-x_v)^2
=\sum_v c_v x_v^2-2\sum_{uv\in E(H)}x_u x_v.
\]
Combining these identities yields
\begin{equation}\label{eq12}
\bm{x}^{\mathrm{T}} A(G) \bm{x} = S^2 - \sum_v (c_v+1)x_v^2+
  \sum_{uv\in E(H)} (x_u-x_v)^2.
\end{equation}
Using $(\lambda_0+c_v+1) x_v^2 = x_v$, we rewrite
$(c_v+1) x_v^2 = x_v - \lambda_0 x_v^2$.
Substituting this into \eqref{eq12}, we obtain the exact identity
\[
\bm{x}^{\mathrm{T}} A(G) \bm{x} = \lambda_0\sum_v x_v^2 + S(S-1)+
\sum_{uv\in E(H)} (x_u-x_v)^2.
\]
The Rayleigh principle implies that
\begin{equation}\label{eq:quadric-form}
\lambda(G)\geq \frac{\bm{x}^{\mathrm{T}} A(G) \bm{x}}{\bm{x}^{\mathrm{T}} \bm{x}}
=\lambda_0+
\frac{S(S-1)+\sum_{uv\in E(H)} (x_u-x_v)^2}{\bm{x}^{\mathrm{T}} \bm{x}}
\geq \lambda_0.
\end{equation}
Since $\lambda_0 = \lambda (T_r(n))$, this proves $\lambda(G)\geq \lambda (T_r(n))$.

It remains to characterize the case of equality. Suppose
$e(G)\geq e(T_r(n))$ and $\lambda(G) = \lambda(T_r(n)) = \lambda_0$.
By \eqref{eq:quadric-form}, we deduce that
$S=1$, and $x_u=x_v$ for every $uv\in E(H)$.
By the equality condition in Claim \ref{claim:convexity}, $S=1$ implies that the multiset $\{d_H(v):v\in V(H)\}$
consists of exactly $p$ copies of $a$ and $q$ copies of $a-1$. Moreover,
$\sum_v d_H(v) = pa+q(a-1)$,
and therefore $e(G)=e(T_r(n))$.

If $b=0$, then $p=0$ and every vertex of $H$ has degree $a-1$. So, every vertex of $G$ has degree $n-1-(a-1)=n-a=(r-1)a$ and thus $G\in\mathcal{F}_r(n)$.
Now assume $1\leq b<r$. Let $X=\{v:d_H(v)=a\}$, $Y=\{v:d_H(v)=a-1\}$.
Then $|X|=p=b(a+1)$ and $|Y|=q=(r-b)a$.
The corresponding values of $x_v$ are
$f(a) = (\lambda_0+a+1)^{-1}$ and $f(a-1) = (\lambda_0+a)^{-1}$,
which are distinct. Hence, the condition $x_u=x_v$ for each $uv\in E(H)$ implies that $H$ has no edge joining $X$ to $Y$. Equivalently, $G$ contains all edges between $X$ and $Y$. Since each vertex in $X$ has degree $a$ in $H$ and no neighbors in $Y$, the graph $H[X]$ is $a$-regular. Similarly, $H[Y]$ is $(a-1)$-regular. Passing back to $G$, we get
$G[X]$ is $|X|-1-a=(b-1)(a+1)$-regular, and
$G[Y]$ is $|Y|-1-(a-1)=(r-b-1)a$-regular.
Thus $G\in\mathcal{F}_r(n)$.

Finally, let $G\in\mathcal{F}_r(n)$. We prove the converse assertions. First suppose $b=0$. Then $G$ is $(r-1)a$-regular on $n=ra$ vertices. Hence $e(G)=n(r-1)a/2=e(T_r(n))$,
and $\lambda(G)=(r-1)a=\lambda(T_r(n))$.
Suppose next that $1\leq b<r$. By Definition \ref{def:family-Fn}, $\overline{G}$ is the disjoint union of an $a$-regular graph on $X$, where $|X|=p$, and an $(a-1)$-regular graph on $Y$, where $|Y|=q$. Therefore, $e(\overline{G})=(pa+q(a-1))/2=m_0$,
and since $e(T_r(n)) = \binom{n}{2} - m_0$, we have $e(G)=e(T_r(n))$.
Define a vector $\bm{z}\in\mathbb{R}^n$ by
\[
z_v=
\begin{cases}
(\lambda_0+a+1)^{-1}, & v\in X,\\
(\lambda_0+a)^{-1}, & v\in Y.
\end{cases}
\]
Again \eqref{eq:equ-lambda0} gives $\sum_v z_v=1$. If $v\in X$, then $v$ has exactly $a$ neighbors in $\overline{G}$, all lying in $X$, so
$(A(G) \bm{z})_v=1-z_v-az_v=\lambda_0 z_v$.
If $v\in Y$, then $v$ has exactly $a-1$ neighbors in $\overline{G}$, all lying in $Y$, so
$(A(G) \bm{z})_v = 1 - z_v - (a-1)z_v = \lambda_0 z_v$.
Thus $A(G) \bm{z} = \lambda_0 \bm{z}$, and hence
$\lambda(G)=\lambda_0=\lambda(T_r(n))$.
The proof is complete.
\end{proof}

\section{Proofs of Theorem \ref{Thm:main1} and Theorem \ref{Thm:join-preserve-spectral-inequality}}

The goal of this section is to present proofs of Theorem \ref{Thm:main1} and Theorem \ref{Thm:join-preserve-spectral-inequality}. We first recall the number of edges of the Tur\'an graph $T_r(n)$.
Let $n = ra + b$, $0\leq b < r$. A simple calculation shows that
\begin{equation}\label{eq:size-Turan-graph}
e(T_r(n)) = \frac{1}{2} \Big( 1 - \frac{1}{r} \Big) n^2 - \frac{1}{2} b \Big( 1 - \frac{b}{r} \Big).
\end{equation}

\subsection{Proof of Theorem \ref{Thm:join-preserve-spectral-inequality}}

\begin{definition}[\cite{McLeman-McNicholas2011}]
Let $G$ be a graph. For $x>\lambda(G)$, the \emph{coronal} $\chi_G(x)$ of $G$ is defined to be the sum of the entries of the matrix $(xI - A(G))^{-1}$, i.e.,
$\chi_G(x) = \bm{1}^{\mathrm{T}} (xI-A(G))^{-1}\bm{1}$.
\end{definition}

\begin{lemma}\label{lem:join-coronal}
Let $G$ be a graph on $n$ vertices, and let $s\geq 1$. The following conclusions hold:
\begin{enumerate}
\item[$(1)$] For $x>\lambda(G)$, denote
$\beta_G(x):= \frac{s}{x} \chi_G(x).$
Then $\beta_G(x)$ is strictly decreasing on $(\lambda(G),\infty)$, and
$\lambda(\overline{K}_s\vee G)$
is the unique $x>\lambda(G)$ satisfying $\beta_G(x)=1$.

\item[$(2)$] We have
\[
\lambda(\overline{K}_s \vee G) \leq \frac{\lambda(G) + \sqrt{\lambda(G)^2 + 4ns}}{2}.
\]
Moreover, if $G$ is regular, then equality holds.
\end{enumerate}
\end{lemma}

\begin{proof}
Let $A$ be the adjacency matrix of $G$, put $\mu:= \lambda(G)$, and let $\bm{x}$ be an eigenvector corresponding
to $\lambda:=\lambda (\overline{K}_s\vee G)$. We may assume $x_v = 1$ for each $v\in V(\overline{K}_s)$. Let $\bm{y}$ be the restriction of $\bm{x}$ on $V(G)$.
So,
\[
\begin{bmatrix}
A & J \\
J^{\mathrm{T}} & O
\end{bmatrix}
\begin{bmatrix}
\bm{y} \\
\bm{1}
\end{bmatrix}
= \lambda
\begin{bmatrix}
\bm{y} \\
\bm{1}
\end{bmatrix},
\]
where $J\in\mathbb{R}^{n\times s}$ is the all-ones matrix. Hence,
\begin{equation}\label{eq:eigen-eigenvector-system}
\begin{cases}
A\bm{y} + s\bm{1} = \lambda\bm{y}, \\
\bm{1}^{\mathrm{T}} \bm{y} = \lambda.
\end{cases}
\end{equation}

(1) Since $\lambda>\lambda(G)$, the matrix $\lambda I-A$ is invertible and $\bm{y} = s(\lambda I-A)^{-1}\bm{1}$. Multiplying by $\bm{1}^{\mathrm{T}}$ and using $\lambda=\bm{1}^{\mathrm{T}} \bm{y}$ gives
$s\cdot\chi_G(\lambda)/\lambda = 1$.

Conversely, if $x>\lambda(G)$ and $s\chi_G(x)/x = 1$, then
$\bm{y} = s (xI-A)^{-1}\bm{1}$ produces a positive eigenvector of $\overline{K}_s\vee G$ with eigenvalue $x$, hence $x=\lambda(\overline{K}_s\vee G)$ by the Perron--Frobenius theorem.

It remains to note monotonicity. Since $\chi_G'(x)=-\bm{1}^{\mathrm{T}} (xI-A)^{-2}\bm{1}<0$, we have
$\beta_G'(x) = s\big(\frac{\chi_G'(x)}{x} - \frac{\chi_G(x)}{x^2}\big)<0$.
Thus, the root is unique.

(2)
Taking the inner product of the first equation in \eqref{eq:eigen-eigenvector-system} with $\bm{y}$, we get
\[
\bm{y}^{\mathrm{T}} A \bm{y} + s \bm{1}^{\mathrm{T}} \bm{y} = \lambda \|\bm{y}\|^2.
\]
Using the second equation $\bm{1}^{\mathrm{T}} \bm{y} = \lambda$, this becomes
$\lambda \|\bm{y}\|^2 = \bm{y}^{\mathrm{T}} A \bm{y} + s(\bm{1}^{\mathrm{T}} \bm{y})^2/\lambda$.

The Rayleigh principle implies that
$\bm{y}^{\mathrm{T}} A \bm{y} \leq \mu \|\bm{y}\|^2$,
and by the Cauchy--Schwarz inequality,
$(\bm{1}^{\mathrm{T}} \bm{y})^2 \leq n \|\bm{y}\|^2$.
Therefore, we have
$\lambda \|\bm{y}\|^2 \leq \mu \|\bm{y}\|^2 + sn \|\bm{y}\|^2/\lambda$. It follows that
$\lambda^2 - \mu\lambda - sn \leq 0$, which implies that
\[
\lambda \leq \frac{\mu + \sqrt{\mu^2 + 4ns}}{2}
= \frac{\lambda(G) + \sqrt{\lambda(G)^2 + 4ns}}{2}.
\]
This proves the upper bound.

Now suppose that $G$ is $d$-regular. Then $d = \lambda(G)$. The partition
$V(\overline{K}_s\vee G) = V(G) \cup V(\overline{K}_s)$
is equitable, and so its quotient matrix is
\[
Q = \begin{bmatrix}
d & s \\
n & 0
\end{bmatrix}.
\]
Hence, $\lambda(\overline{K}_s\vee G)$ is equal to
the largest eigenvalue of $Q$, namely
$(d + \sqrt{d^2 + 4ns})/2$.
Thus equality holds when $G$ is regular.
\end{proof}

\begin{lemma}\label{lem:coronal-bound}
Let $G$ be a graph on $n$ vertices with average degree $d$. Then, for every $x>\lambda(G)$,
\[
\chi_G(x) \leq \frac{n(x+d)}{x^2 - \lambda(G)^2}.
\]
\end{lemma}

\begin{proof}
Let all eigenvalues of $A(G)$ be $\lambda_1,\lambda_2,\ldots,\lambda_n$, and let $\bm{u}_1,\bm{u}_2,\ldots,\bm{u}_n$ be an orthonormal eigen-basis.
Write $\bm{1} = \sum_{i=1}^n \alpha_i\bm{u}_i$, where $\sum_{i=1}^n \alpha_i^2 = n$. Then
\begin{equation}\label{eq:sum-inner-prod}
\sum_{i=1}^n \langle \bm{1}, \bm{u}_i\rangle^2
= \sum_{i=1}^n \langle \alpha_1\bm{u}_1+\cdots+\alpha_n\bm{u}_n, \bm{u}_i \rangle^2 = \sum_{i=1}^n \alpha_i^2 = n.
\end{equation}
Moreover, we have
$A(G)\bm{1} = \sum_{i=1}^n \alpha_i A(G) \bm{u}_i = \sum_{i=1}^n \alpha_i\lambda_i\bm{u}_i.$
It follows that $\bm{1}^{\mathrm{T}} A(G) \bm{1}
= \langle \bm{1}, A(G)\bm{1}\rangle
= \sum_{i=1}^n \alpha_i^2\lambda_i$.
Since $\alpha_i = \langle \bm{1}, \bm{u}_i\rangle$, $i=1,2,\ldots,n$, we have
\begin{equation}\label{eq:sum-average-degree}
\sum_{i=1}^n \langle \bm{1}, \bm{u}_i\rangle^2 \lambda_i
= \bm{1}^{\mathrm{T}} A(G) \bm{1}
= 2e(G) = nd.
\end{equation}
In view of $(xI-A(G))^{-1} \bm{u}_i = (x-\lambda_i)^{-1} \bm{u}_i$, and $\bm{1} = \sum_{i=1}^n \alpha_i \bm{u}_i$, we have
\begin{equation}\label{eq:chi-G-angles}
\bm{1}^{\mathrm{T}} (xI-A(G))^{-1}\bm{1}
= \bigg\langle \sum_{j=1}^n \alpha_j\bm{u}_j, \sum_{i=1}^n \frac{\alpha_i}{x-\lambda_i} \bm{u}_i\bigg\rangle
= \sum_{i=1}^n \frac{\alpha_i^2}{x-\lambda_i}
= \sum_{i=1}^n \frac{\langle \bm{1}, \bm{u}_i\rangle^2}{x-\lambda_i}.
\end{equation}
Finally, since $|\lambda_i|\leq \lambda(G)$, one can check for $x > \lambda(G)$,
$\frac{1}{x-\lambda_i}
\leq\frac{x+\lambda_i}{x^2-\lambda(G)^2}$.
It follows from \eqref{eq:chi-G-angles}, \eqref{eq:sum-inner-prod} and \eqref{eq:sum-average-degree} that
\[
\chi_G(x) =\sum_{i=1}^n \frac{\langle \bm{1}, \bm{u}_i\rangle^2}{x-\lambda_i}\leq
\sum_{i=1}^n \langle \bm{1}, \bm{u}_i\rangle^2 \frac{x+\lambda_i}{x^2-\lambda(G)^2}
=\frac{nx+nd}{x^2-\lambda(G)^2},
\]
which proves the lemma.
\end{proof}

\begin{lemma}\label{lem:turan-quotient}
Let $r\geq 1$ and $s\geq 1$ be integers.
Denote $\lambda = \lambda(T_r(n))$,
$\mu = \lambda (\overline{K}_s\vee T_r(n))$, and
$d_0 = 2\big(e(T_r(n)) - 1\big)/n$. Then
\begin{equation}\label{eq:turan-quotient}
\frac{n(\mu + d_0)}{\mu^2 - \lambda^2}<\frac{\mu}{s}.
\end{equation}
\end{lemma}

\begin{proof}
The assertion is clear when $r=1$. Hence, in the following we assume $r\geq 2$. Write $n=ra+b$, where $0\leq b < r$ and $a\geq 1$.
We first consider the case $b=0$. Then $T_r(n)$ is regular of degree
$\lambda = (r-1)a$. Applying Lemma \ref{lem:complete-r-partite} to
$\overline{K}_s\vee T_r(n)$, we obtain
\begin{equation}\label{eq:equation-for-mu}
\frac{s}{\mu + s} + \frac{ra}{\mu + a} = 1.
\end{equation}
Since $\lambda = (r-1)a$ and $n = ra$, we have $a = n - \lambda$. Substituting this into \eqref{eq:equation-for-mu} gives
$\frac{s}{\mu + s} + \frac{n}{\mu + n - \lambda} = 1,$
which yields $\mu^2- \lambda \mu - sn = 0$, and hence
$\frac{\mu}{s} = \frac{n}{\mu - \lambda}$.
Moreover, $d_0= \lambda - 2/n$. Hence,
\[
\frac{\mu}{s}-\frac{n(\mu + d_0)}{\mu^2 - \lambda^2}
= \frac{n}{\mu - \lambda} - \frac{n(\mu + \lambda -2/n)}{\mu^2 - \lambda^2}= \frac{2}{\mu^2 - \lambda^2} > 0.
\]
Thus, \eqref{eq:turan-quotient} holds when $b=0$. Henceforth,
assume that $1\leq b<r$.

For brevity, set
$C:=ra(a+1)$, $L:=n-2a-1$, $M:=a(a+1)(r-1)$.
Define
\[
U(x):= \frac{b(a+1)}{x+a+1} + \frac{(r-b)a}{x+a}
= \frac{nx+C}{(x+a)(x+a+1)},
\]
where $x> \lambda$.
Since the graph $\overline{K}_s\vee T_r(n)$ has $b$ parts of size $a+1$, $r-b$ parts of size $a$, and one part of size $s$, Lemma \ref{lem:complete-r-partite} yields
$\frac{s}{\mu + s} + U(\mu) = 1$. Hence, we have $U(\mu) = \frac{\mu}{\mu + s}$. Set
\[
\Phi(x):=U(x)(x^2- \lambda^2)-n(x+d_0)(1-U(x)).
\]
The desired inequality \eqref{eq:turan-quotient} is equivalent to $\Phi(\mu) > 0$.

In what follows, we shall simplify $\Phi(x)$, and then prove $\Phi(\mu) > 0$. To this end,
note that $T_r(n)$ has $b$ parts of size $a+1$ and $r-b$ parts of size $a$. Applying Lemma \ref{lem:complete-r-partite} again gives $U(\lambda) = 1$. Equivalently,
$\lambda^2 - L\lambda - M = 0$. Hence, $\lambda > L$.
By simple algebra, we see
\begin{equation}\label{eq:1-Ux}
1-U(x)=\frac{x^2-Lx-M}{(x+a)(x+a+1)}
=\frac{(x- \lambda)(x+ \lambda -L)}{(x+a)(x+a+1)},
\end{equation}
where the last equality follows from $\lambda^2 - L\lambda  - M = 0 $.
Let $d_T=2e(T_r(n))/n$ be the average degree of $T_r(n)$. By \eqref{eq:size-Turan-graph}, and using $b = n - ra$, we obtain
\begin{equation}\label{eq:d_T}
d_T
= n - \frac{n}{r} - \frac{b}{n} + \frac{b^2}{rn} = n - \frac{n}{r} - \frac{n-ra}{n} + \frac{(n-ra)^2}{rn} = L + \frac{C}{n},
\end{equation}
and therefore
\begin{equation}\label{eq:d0-formula}
d_0 =d_T - \frac{2}{n} = L + \frac{C}{n} - \frac{2}{n}.
\end{equation}
Recall that $\Phi(x):=U(x)(x^2- \lambda^2)-n(x+d_0)(1-U(x))$. We have
\begin{align*}
\Phi(x)
& =
\frac{(nx + C) (x^2 - \lambda^2)}{(x+a)(x+a+1)}
- \frac{(x- \lambda)(x+ \lambda -L)}{(x+a)(x+a+1)} \cdot n(x+d_0) \\
& = \frac{x-\lambda}{(x+a)(x+a+1)} \big((nx+C)(x+\lambda) - (nx + nd_0)(x + \lambda - L)\big).
\end{align*}
From \eqref{eq:d0-formula} we obtain $C = nd_0-nL+2$. Substituting this into the equation above gives
$\Phi(x) = \frac{x-\lambda}{(x+a)(x+a+1)}
\big (2(x+\lambda)+n(d_0-\lambda)L\big).$
Therefore, to complete the proof, it remains to prove
$2(\mu + \lambda) + n(d_0 - \lambda)L > 0$.

Since $\mu > \lambda$, it suffices to show that
$4\lambda + nL (d_0 - \lambda) > 0$.
To this end, set $\delta:=\lambda-d_T\geq 0$.
Let $p(x):= x^2 - Lx - M$. Then $p(\lambda)=0$. By \eqref{eq:d_T}, a short calculation gives
\begin{align*}
p(d_T)
& = d_T^2 - L\cdot d_T - M
= \Big(L + \frac{C}{n}\Big)^2 - L\cdot \Big(L + \frac{C}{n}\Big) - M = \frac{LC}{n} + \frac{C^2}{n^2} - M.
\end{align*}
Substituting $C = ra(a+1)$, $L = n-2a-1$, and $M = a(a+1)(r-1)$, we get
\begin{align}
p(d_T) & = \frac{ra(a+1) L}{n} + \frac{r^2a^2(a+1)^2}{n^2} - a(a+1) (r-1) \nonumber \\
& = -\frac{a(a+1)}{n^2} \big( (r-1)n^2 - rn L - r^2 a(a+1) \big) \nonumber \\
& = -\frac{a(a+1)}{n^2} \big((r-1)n^2 - r(n-2a-1)n - r^2a(a+1) \big) \nonumber \\
& = -\frac{a(a+1)}{n^2} \big(-n^2 + r(2a+1) n - r^2a (a+1)\big) \nonumber \\
& = -\frac{a(a+1)b(r-b)}{n^2}, \label{eq:p(d_T)}
\end{align}
where the last equality follows from $n = ra+b$.
On the other hand, since $\lambda^2 - L\lambda = M$, we have
\[
p(d_T) = d_T^2 - Ld_T - M
= (d_T^2 - \lambda^2) - (
Ld_T - L\lambda)
= (d_T - \lambda) (\lambda +d_T - L).
\]
Combining with \eqref{eq:p(d_T)} we have
\begin{equation}\label{eq:delta-formula}
\delta = \lambda -d_T
=\frac{-p(d_T)}{\lambda +d_T-L}
=\frac{a(a+1)b(r-b)}{n^2(\lambda + C/n)}.
\end{equation}
Since $d_0=d_T-2/n$, and $d_T = \lambda - \delta$, we have
\begin{equation}\label{eq:min-bracket}
4\lambda + n(d_0 - \lambda)L = 4\lambda + nL \Big(d_T - \frac{2}{n} - \lambda\Big) = 2(2\lambda - L) - nL\delta.
\end{equation}
If $L=0$, then $4\lambda + nL(d_0 - \lambda) = 4\lambda > 0$. If $L>0$, then
$\lambda > L$ by the fact $\lambda^2 - L\lambda - M = 0$.
Moreover, since $L = n-2a-1>0$, we have $r\geq 3$.
It follows from \eqref{eq:delta-formula} that
\begin{equation}\label{eq:nLdelta-bound}
nL\delta = \frac{L\,a(a+1)b(r-b)}{n\lambda + C}
\le \frac{a(a+1)b(r-b)}{n}.
\end{equation}
Since $r\geq 3$, we also have
$b(r-b)\leq r^2/4$, $a+1\leq 2a$,
$n\geq ra$, and $L=(r-2)a+b-1\geq (r-2)a$.
Hence
\begin{equation}\label{eq:elementary-bound}
\frac{a(a+1)b(r-b)}{n} \leq \frac{ar}{2}
\leq 2(r-2)a \leq 2L.
\end{equation}
Combining~\eqref{eq:min-bracket}--\eqref{eq:elementary-bound} yields
\[
4\lambda + n(d_0 - \lambda)L \geq 2(2\lambda - L)-2L
= 4(\lambda - L) > 0,
\]
because $\lambda > L$. Hence $\Phi(\mu) > 0$, and therefore~\eqref{eq:turan-quotient}.
\end{proof}

\begin{lemma}\label{lem:F-join-equality}
Let $r\geq 2$ and $s\geq 1$. If
$F\in\mathcal F_r(n)$, then
\[
 \lambda(\overline K_s\vee F)=\lambda(\overline K_s\vee T_r(n)).
\]
\end{lemma}

\begin{proof}
Write $n=ra+b$, where $a\geq 1$ and $0\leq b<r$, and let $Z$ denote
$V(\overline K_s)$.
If $b=0$, then both $F$ and $T_r(n)$ are $(r-1)a$-regular.  Hence the
partitions $V(F)\sqcup Z$ and $V(T_r(n))\sqcup Z$ are equitable and share the
same quotient matrix
\[
 \begin{bmatrix}
 (r-1)a & s \\
 n & 0
 \end{bmatrix}.
\]
Hence, $\overline{K}_s\vee F$ and $\overline{K}_s\vee T_r(n)$ have the same spectral
radius.

Now assume $1\leq b<r$. Let $V(F)=X\sqcup Y$ be the partition from
Definition~\ref{def:family-Fn}; thus
$|X|=b(a+1)$, $|Y|=(r-b)a$, $X$ is completely joined to $Y$, $F[X]$ is
$(b-1)(a+1)$-regular, and $F[Y]$ is $(r-b-1)a$-regular.  In
$\overline K_s\vee F$, the partition $X\sqcup Y\sqcup Z$ is equitable
with quotient matrix
\[
Q=
\begin{bmatrix}
 (b-1)(a+1) & (r-b)a & s \\
 b(a+1) & (r-b-1)a & s \\
 b(a+1) & (r-b)a & 0
\end{bmatrix}.
\]
The same quotient matrix is obtained for
$\overline K_s\vee T_r(n)$ by taking $X$ to be the union of the $b$ parts
of size $a+1$ and $Y$ to be the union of the remaining $r-b$ parts of
size $a$. It follows that $ \lambda(\overline K_s\vee F)=\lambda(\overline K_s\vee T_r(n))$.
\end{proof}

\begin{proof}[Proof of Theorem \ref{Thm:join-preserve-spectral-inequality}]
The claim holds trivially for $r=1$; we therefore restrict our attention to the case $r\geq 2$. For brevity, set $\lambda:= \lambda(T_r(n))$, $\mu:=\lambda(\overline{K}_s\vee T_r(n))$, $\eta:=\lambda(G)$, and $\rho:= \lambda (\overline{K}_s\vee G)$.

(1) If $\eta\geq\mu$, then immediately $\eta\geq \mu > \lambda$,
so there is nothing more to prove. Hence assume from now on that
$\eta < \mu$.

Since $\rho\geq \mu$ and $\mu > \eta$, Lemma \ref{lem:join-coronal} implies $\chi_G(\mu)\geq \mu/s$.
By Lemma \ref{lem:coronal-bound}, we have
$\chi_G(\mu)\leq n(\mu + d)/(\mu^2-\eta^2)$, where $d$ is the average degree of $G$.
Consequently,
\begin{equation}\label{eq:key-direct}
\frac{\mu}{s} \leq \frac{n(\mu +d)}{\mu^2-\eta^2}.
\end{equation}
Set $d_0:=2(e(T_r(n))-1)/n$. There are two cases.

\smallskip
\noindent\textbf{Case 1}: $d\leq d_0$.
It follows from \eqref{eq:key-direct} that
$\mu^2-\eta^2 \leq sn(\mu + d)/\mu
\leq sn(\mu + d_0)/\mu$.
Hence, $\eta^2\geq \mu^2 - sn(\mu + d_0)/\mu$.
By Lemma \ref{lem:turan-quotient},
$\mu^2 - sn(\mu + d_0)/\mu > \lambda^2$.
Thus $\eta>\lambda$.

\smallskip
\noindent\textbf{Case 2}: $d > d_0$.
Then $e(G) > e(T_r(n)) - 1$, and therefore $e(G)\geq e(T_r(n))$. Then
$\eta = \lambda(G)\geq \lambda(T_r(n)) = \lambda$.

Hence, we have $\lambda(G)\geq\lambda(T_r(n))$. Moreover, if $\rho\geq \mu$ and $\eta=\lambda$,
then $d>d_0$, which implies $e(G)=e(T_r(n))$.
By Theorem \ref{Thm:Prob1}, we see $G\in\mathcal{F}_r(n)$.
By Lemma~\ref{lem:F-join-equality}, this further implies
$\rho=\mu$.

(2) If $\rho > \mu$, then part (1) implies that $\eta\geq\lambda$. If $\eta = \lambda$, then the ``moreover" part of (1) gives $\rho=\mu$, a contradiction.
Hence, $\eta>\lambda$.

(3) is exactly (1) in the case $\rho = \mu$.
\end{proof}

\subsection{Proof of Theorem \ref{Thm:main1}}

\begin{lemma}\label{lem:F-local}
Let $F\in\mathcal F_r(n)$. If $F\not\cong T_r(n)$, then there is a
vertex $u\in V(F)$ such that
\[
e(F[N_F(u)])>e(T_{r-1}(d_F(u))).
\]
\end{lemma}

\begin{proof}
Write $n=ra+b$, where $a\geq 1$ and $0\leq b<r$. Put
$H = \overline F$.

First suppose that $b=0$. Then $n=ra$, and by the definition of
$\mathcal F_r(n)$, the graph $F$ is $(r-1)a$-regular. Hence $H$ is
$(a-1)$-regular.  Fix a vertex $u$ and define
\[
A_u:= \{u\}\cup N_H(u), \qquad
B_u:= V(F)\setminus A_u=N_F(u).
\]
Then $|A_u|=a$ and $|B_u|=(r-1)a=d_F(u)$. Since $H$ is
$(a-1)$-regular,
\[
\sum_{z\in B_u} d_H(z)=(r-1)a(a-1) =2e_H(B_u)+e_H(A_u,B_u).
\]
Consequently, we obtain that $e_H(B_u) = ((r-1)a(a-1) - e_H(A_u,B_u))/2$.
Thus,
\begin{align*}
e_F(B_u)
& = \binom{(r-1)a}{2} - e_H(B_u) \\
& = \binom{(r-1)a}{2} - \frac{(r-1)a(a-1)}2
+ \frac{e_H(A_u,B_u)}2 \\
& = \binom{r-1}{2}a^2 + \frac{e_H(A_u,B_u)}2 \\
& = e(T_{r-1}((r-1)a)) + \frac{e_H(A_u,B_u)}2.
\end{align*}
Hence $e(F[N_F(u)]) > e(T_{r-1}(d_F(u)))$
whenever $e_H(A_u,B_u) > 0$.

Now suppose $e_H(A_u,B_u) = 0$ for every $u$. Then each $A_u$ forms a component of
$H$. Since $|A_u|=a$ and every vertex of $A_u$ has degree $a-1$ in $H$,
it follows that $H[A_u]$ must be a copy of $K_a$. Hence $H$ is the disjoint union
of $r$ copies of $K_a$, and therefore $F\cong K_{a,a,\ldots,a} \cong T_r(n)$.
This contradicts the assumption $F\not\cong T_r(n)$. The desired inequality follows
in the case $b=0$.

Next, suppose that $1\leq b<r$. Let
\[
V(F) = X\sqcup Y, \qquad |X|=b(a+1), \qquad |Y|=(r-b)a
\]
be the partition from the definition of $\mathcal F_r(n)$. The set
$X$ is completely joined to $Y$ in $F$, so the complement $H = \overline{F}$
has no edges between $X$ and $Y$. Moreover,
$H[X]$ is $a$-regular and $H[Y]$ is $(a-1)$-regular.

Take first a vertex $u\in X$. Put $A_u = \{u\}\cup N_{H[X]}(u)$ and $B_u = X\setminus A_u$.
Then $|A_u| = a+1$, $|B_u|=(b-1)(a+1)$, and
\[
N_F(u) = B_u\sqcup Y, \qquad
d_F(u) = (b-1)(a+1)+(r-b)a=(r-1)a+b-1.
\]
Since $H[X]$ is $a$-regular, we have $|B_u|a = 2e_{H[X]}(B_u) + e_{H[X]}(A_u,B_u)$. Thus
\[
e_{H[X]}(B_u) = (b-1)\binom{a+1}{2} - \frac{e_{H[X]}(A_u,B_u)}2.
\]
Also, $H[Y]$ is $(a-1)$-regular, giving $e_{H[Y]}(Y) = (r-b)\binom{a}{2}$.
The complement of $F[N_F(u)]$ is exactly
$H[B_u]\cup H[Y]$. Hence,
\begin{align*}
e(F[N_F(u)])
& = \binom{d_F(u)}2 - e_{H[X]}(B_u) - e_{H[Y]}(Y) \\
& = \binom{d_F(u)}2 - (b-1)\binom{a+1}{2} - (r-b)\binom a2
+ \frac{e_{H[X]}(A_u,B_u)}2.
\end{align*}
But $d_F(u)=(r-1)a+b-1$, so $T_{r-1}(d_F(u))$ has $b-1$ parts of size
$a+1$ and $r-b$ parts of size $a$. Therefore,
\[
e(T_{r-1}(d_F(u))) = \binom{d_F(u)}2 - (b-1)\binom{a+1}{2} - (r-b)\binom a2.
\]
Combining the last two displayed equations gives
\[
e(F[N_F(u)]) = e(T_{r-1}(d_F(u))) + \frac{e_{H[X]}(A_u,B_u)}2.
\]
Hence any $u\in X$ with $e_{H[X]}(A_u,B_u)>0$ gives the desired
inequality.

Similarly, take a vertex $u\in Y$ and put $C_u = \{u\}\cup N_{H[Y]}(u)$,
$D_u = Y\setminus C_u$.
Then $|C_u|=a$, $|D_u|=(r-b-1)a$, and
\[
N_F(u)=X\sqcup D_u,\qquad d_F(u)=b(a+1)+(r-b-1)a=(r-1)a+b.
\]
Since $H[Y]$ is $(a-1)$-regular, we find
\[
e_{H[Y]}(D_u) = (r-b-1)\binom a2-\frac{e_{H[Y]}(C_u,D_u)}2.
\]
Also, $e_{H[X]}(X) = b\binom{a+1}{2}$.
The graph $T_{r-1}(d_F(u))$ has $b$ parts of size $a+1$ and
$r-b-1$ parts of size $a$; when $b=r-1$, this simply means that all
$r-1$ parts have size $a+1$.  Hence
\[
e(T_{r-1}(d_F(u))) = \binom{d_F(u)}2 - b\binom{a+1}{2} - (r-b-1)\binom a2,
\]
and a similar complement count gives
\[
e(F[N_F(u)]) = e(T_{r-1}(d_F(u))) + \frac{e_{H[Y]}(C_u,D_u)}2.
\]
Thus any vertex $u\in Y$ with $e_{H[Y]}(C_u,D_u)>0$ also gives the desired
inequality.

It remains to show that such a vertex exists under the assumptions that
$e_{H[X]}(A_u,B_u)=0$ for every $u\in X$ and $e_{H[Y]}(C_u,D_u)=0$ for every $u\in Y$.
Since $e_{H[X]}(A_u,B_u)=0$ for every $u\in X$,
it follows that $H[X]$ is the disjoint union of $b$ copies of $K_{a+1}$. Likewise, the condition $e_{H[Y]}(C_u,D_u)=0$ for every $u\in Y$ implies that
$H[Y]$ is the disjoint union of $r-b$ copies of $K_a$. Hence
$H\cong bK_{a+1}\sqcup (r-b)K_a$, and therefore $F\cong T_r(n)$, contradicting $F\not\cong T_r(n)$.
\end{proof}

\begin{proof}[Proof of Theorem \ref{Thm:main1}]
We shall use the following spectral symmetrization from \cite{G96}. Let $\bm{x}$ be a nonnegative eigenvector corresponding to the spectral radius of $G$, and let $v$ be a vertex with
$x_v=\max\{x_u:u\in V(G)\}$. Define
\[
N=N_G(v),\qquad S=V(G)\setminus N,
   \qquad s=|S|,
   \qquad H=G[N].
\]
Notice that $v\in S$ and
$x_v > 0$. We claim that $\lambda(G)\leq \lambda(\overline{K}_s\vee H)$.
Indeed, the eigenvalue equation at $v$ gives
$\sum_{z\in N} x_z = \lambda(G) x_v$.
For each $u\in S$, using $x_u\leq x_v$ and the eigenvalue equation at
$u$, we get
\begin{equation}\label{eq:eigen-equation-inequality}
\sum_{z\in N}x_z=\lambda(G)x_v\geq \lambda(G)x_u
=\sum_{z\in N_G(u)\cap N}x_z+
  \sum_{z\in N_G(u)\cap S}x_z .
\end{equation}
Now consider the difference $\bm{x}^{\mathrm T} A(\overline{K}_s\vee H)\bm{x} - \bm{x}^{\mathrm T} A(G)\bm{x}$. Since
\begin{equation}\label{eq:difference-quadric-form}
\frac{1}{2} \bm{x}^{\mathrm T} \big(A(\overline{K}_s\vee H) -A(G)\big)\bm{x} =
\sum_{\{uv\}\in E(\overline{K}_s\vee H)} x_ux_v
- \sum_{uv\in E(G)} x_ux_v,
\end{equation}
we only need to compare the edge sets of $\overline{K}_s\vee H$ and $G$.
On $N$, the two graphs coincide, so edges inside $N$ make no
contribution to the difference. Between $S$ and $N$, the graph $\overline{K}_s\vee H$ contains all possible
edges, while $G$ contains precisely those edges from $u\in S$ to
$N_G(u)\cap N$. Hence the total contribution from the pairs with one
endpoint in $S$ and the other in $N$ is
\[
\sum_{u\in S} x_u\Bigg(\sum_{z\in N} x_z -
\sum_{z\in N_G(u)\cap N} x_z\Bigg).
\]
Moreover, every edge of $G[S]$ is deleted in passing
from $G$ to $\overline{K}_s\vee H$; these deleted edges contribute
$- \sum_{uz\in E(G[S])} x_ux_z$.
Combining these contributions gives
\begin{align}
\frac{1}{2}\big(\bm{x}^{\mathrm T}A(\overline{K}_s\vee H)\bm{x} - \bm{x}^{\mathrm T} A(G)\bm{x}\big)
& = \sum_{u\in S} x_u\Bigg(\sum_{z\in N} x_z -
\sum_{z\in N_G(u)\cap N} x_z\Bigg)
- \sum_{uz\in E(G[S])} x_ux_z \nonumber \\
& \geq \sum_{u\in S} x_u\sum_{z\in N_G(u)\cap S} x_z
- \sum_{uz\in E(G[S])} x_ux_z \nonumber \\
& = \sum_{uz\in E(G[S])} x_ux_z \geq 0, \label{eq:difference-Rayleigh-quotient}
\end{align}
where the inequality in the second line follows from \eqref{eq:eigen-equation-inequality}.
By the Rayleigh principle, we conclude $\lambda(\overline{K}_s\vee H)\geq\lambda(G)$.

We first consider the `furthermore' part of Theorem \ref{Thm:main1}. Assume $\lambda(G)>\lambda(T_r(n))$.
Since any complete $r$-partite graph on $n$ vertices has spectral radius at most $\lambda (T_r(n))$,
we have
\begin{equation}\label{eq:inequality-chain}
\lambda (\overline{K}_s\vee T_{r-1}(d(v))) \leq \lambda (T_r(n)) < \lambda (G) \leq \lambda(\overline{K}_s\vee H).
\end{equation}
Thus $\lambda (\overline{K}_s\vee H) > \lambda (\overline{K}_s\vee T_{r-1}(d(v))$. By Theorem \ref{Thm:join-preserve-spectral-inequality} (2), it follows that $\lambda(G[N(v)]) > \lambda(T_{r-1}(d(v)))$, finishing the proof of second part.

For the first part of Theorem \ref{Thm:main1}, it remains to consider the case $\lambda(G) = \lambda(T_r(n))$ and $G\not\cong T_r(n)$. Equivalently, we must show that under these conditions, there is a vertex $v$ such that $\lambda(G[N_G(v)]) > \lambda(T_{r-1}(d_G(v)))$.

If $G$ is disconnected, let $C$ be a component
with $\lambda(C)=\lambda(G)$, and set $m=|V(C)|<n$.
Then $\lambda(C)=\lambda(T_r(n))>\lambda(T_r(m))$.
Applying the already proved `furthermore' part to $C$ gives a vertex
$u\in V(C)$ such that $\lambda(C[N_C(u)])>\lambda(T_{r-1}(d_C(u)))$.
Since $C$ is a component of $G$, this is exactly the required conclusion
for $G$. Hence we may assume that $G$ is connected; in particular, the
Perron vector $\bm{x}$ is positive.

Return to the maximum-weight vertex $v$ fixed above. If $\lambda(\overline{K}_s\vee H) > \lambda(\overline{K}_s\vee T_{r-1}(d_G(v)))$, then
Theorem \ref{Thm:join-preserve-spectral-inequality} (2) gives
$\lambda(H)>\lambda(T_{r-1}(d_G(v)))$, and we are done. Hence we may assume $\lambda(\overline{K}_s\vee H)\leq\lambda(\overline{K}_s\vee T_{r-1}(d_G(v)))$.
Analogously to \eqref{eq:inequality-chain}, we obtain
\begin{equation}\label{eq:inequality-chain-2}
\lambda(\overline{K}_s\vee T_{r-1}(d(v))) \leq \lambda(T_r(n)) = \lambda(G) \leq\lambda(\overline{K}_s\vee H) \leq\lambda(\overline{K}_s\vee T_{r-1}(d(v))),
\end{equation}
this forces $\lambda(\overline{K}_s\vee H) = \lambda(\overline{K}_s\vee T_{r-1}(d(v)))$.
By Theorem \ref{Thm:join-preserve-spectral-inequality}, we see $\lambda(H)\geq\lambda(T_{r-1}(d(v)))$. If $\lambda(H)>\lambda(T_{r-1}(d(v)))$, then
the desired conclusion holds. Thus assume $\lambda(H)=\lambda(T_{r-1}(d(v)))$. Then Theorem \ref{Thm:join-preserve-spectral-inequality} (3) implies $H\in\mathcal F_{r-1}(d_G(v))$.
Moreover, equality in \eqref{eq:inequality-chain-2} together with the equality
case of the spectral Tur\'an theorem yields
\begin{equation}\label{eq:unique-spectral-turan}
   \overline{K}_s\vee T_{r-1}(d_G(v))\cong T_r(n).
\end{equation}
From $H\in\mathcal F_{r-1}(d_G(v))$ and \eqref{eq:unique-spectral-turan}, it follows
that
\begin{equation}\label{eq:join-in-F-main1}
\overline K_s\vee H\in\mathcal F_r(n).
\end{equation}
Indeed, write $n=ra+b$, where $a\geq 1$ and $0\leq b<r$. By
\eqref{eq:unique-spectral-turan}, the set $S$ has size either $a$ or $a+1$.
If $b=0$, then necessarily $s=a$ and $d_G(v)=(r-1)a$. Since
$H\in\mathcal F_{r-1}((r-1)a)$, the graph $H$ is $(r-2)a$-regular; hence
every vertex of $\overline K_s\vee H$ has degree $(r-1)a$, and so
$\overline K_s\vee H\in\mathcal F_r(n)$.

Assume next that $1\leq b<r$. If $s=a+1$, then
$d_G(v)=(r-1)a+b-1$. If $b=1$, then by the definition of $\mathcal{F}_r(n)$ we have $\overline{K}_s\vee H\in\mathcal{F}_r(n)$; if $b>1$, since $H\in\mathcal F_{r-1}(d_G(v))$, there exists a partition $V(H) = X\sqcup Y$ with $|X|=(b-1)(a+1)$, $|Y|=(r-b)a$, where
$H[X]$ is $(b-2)(a+1)$-regular and $H[Y]$ is $(r-b-1)a$-regular.
One can verify that $\overline{K}_s\vee H$ belongs to $\mathcal{F}_r(n)$ with partition $(S\sqcup X)\sqcup Y$.
If $s=a$, then $d_G(v)=(r-1)a+b$. A similar argument shows $\overline{K}_s\vee H\in\mathcal{F}_r(n)$.

Now, since $\lambda(G)=\lambda(\overline K_s\vee H)$, equality must hold in
\eqref{eq:difference-Rayleigh-quotient}. This first implies that $G[S]$ has no edges, and then for each $u\in S$,
\[
\sum_{z\in N} x_z = \sum_{z\in N_G(u)\cap N} x_z.
\]
It follows that each vertex of $S$ is
adjacent to every vertex of $N$. Thus $G=\overline{K}_s\vee H$.
By \eqref{eq:join-in-F-main1}, we have $G=\overline K_s\vee H\in\mathcal F_r(n)$. Since $G\not\cong T_r(n)$, Lemma~\ref{lem:F-local} guarantees a vertex $u$ such that
$e(G[N_G(u)])>e(T_{r-1}(d_G(u)))$.
Applying Corollary~\ref{coro:Prob1} yields
$\lambda(G[N_G(u)])>\lambda(T_{r-1}(d_G(u)))$.
This completes the proof.
\end{proof}

\section{Proof of Theorem \ref{Thm:Prob2}}
\label{Sec:4}

In this section, we present a proof of Theorem \ref{Thm:Prob2}.
Let $k\geq 2$ be an integer, and let $G^{\vee k}$ denote the $k$-fold join of $G$, defined as the graph obtained by taking $k$ vertex-disjoint copies of $G$ and adding all edges between vertices belonging to distinct copies.
For further results on the spectral properties of $G^{\vee k}$, we refer the reader to \cite{Cardoso-Gomes-Pinheiro2022}. For a symmetric matrix $M$, let $\lambda_{\min}(M)$ denote the smallest eigenvalue of $M$.

\begin{proof}[Proof of Theorem \ref{Thm:Prob2}]
Let $H_k$ denote the $k$-fold join of $G$, and set $d:=d(G)$ and $s:=n-d> 0$. Obviously, $|V(H_k)| = kn$, $\omega(H_k) = k\omega(G)$, and
the number of edges of $H_k$ is
$e(H_k)=ke(G)+\binom{k}{2}n^2$.
Consequently, the average degree of $H_k$ is
\begin{equation}\label{eq:avg-degree-join}
d(H_k)=d+(k-1)n = kn-(n-d) = kn-s.
\end{equation}

We next claim a key spectral fact: $\lambda_{\min}(H_k)$ is bounded independently of
$k$. Let $A=A(G)$, and let $J_n$ and $J_k$ denote the all-ones matrices of
orders $n$ and $k$, respectively. The adjacency matrix of $H_k$ can be expressed as
$A(H_k) = I_k\otimes A + (J_k-I_k)\otimes J_n$, where $P\otimes Q$ is the
Kronecker product of matrices $P$ and $Q$.
Writing $J_k = \bm{1}_k \bm{1}_k^{\mathrm{T}}$ and $J_n = \bm{1}_n \bm{1}_n^{\mathrm{T}}$, we obtain
\begin{align*}
A(H_k) & = I_k \otimes A + (\bm{1}_k \bm{1}_k^{\mathrm{T}} - I_k) \otimes (\bm{1}_n \bm{1}_n^{\mathrm{T}}) \\
& = I_k \otimes A + (\bm{1}_k\bm{1}_k^{\mathrm{T}}) \otimes (\bm{1}_n\bm{1}_n^{\mathrm{T}}) - I_k \otimes J_n \\
& = I_k\otimes (A - J_n) + (\bm{1}_k \otimes \bm{1}_n)(\bm{1}_k \otimes \bm{1}_n)^{\mathrm{T}}.
\end{align*}
Since $(\bm{1}_k \otimes \bm{1}_n)(\bm{1}_k \otimes \bm{1}_n)^{\mathrm{T}}$ is positive semidefinite, it follows from Weyl's inequality that
\[
0\geq\lambda_{\min} (H_k) \geq \lambda_{\min} (I_k\otimes (A - J_n))
= \lambda_{\min} (A-J_n).
\]
Hence, $\lambda_{\min}(H_k)$ is bounded independently of $k$.
Now apply \eqref{eq:least-eigenvalue-bound} to $H_k$. Using
\eqref{eq:avg-degree-join}, we get
\[
k\omega(G)\geq
1 + \frac{kn\cdot d(H_k)}{\big(kn-d(H_k)\big) \big(d(H_k) -\lambda_{\min}(H_k)\big)} = 1 + \frac{(kn-s)kn}{s\big(kn-s-\lambda_{\min}(H_k)\big)}.
\]
Dividing both sides by $k$ gives
\begin{equation}\label{eq:divide-by-k}
\omega(G)\geq\frac{1}{k} + \frac{n(kn-s)}{s\big(kn-s-\lambda_{\min}(H_k)\big)}.
\end{equation}
Taking $k\to\infty$ in \eqref{eq:divide-by-k} yields
$\omega(G)\geq n/(n-d)$.

Finally, if $e(G) > (1-1/r) n^2/2$, then
$d> (1-1/r) n$, and therefore $n/(n-d) > r$.
Hence $\omega(G)\geq r+1$, which implies that $G$ contains $K_{r+1}$ as a subgraph.
\end{proof}

\section{Concluding remarks}

The results above show that several classical implications between Tur\'an-type statements and spectral Tur\'an-type statements are sharper than one might first expect.  We close with two related directions.  The first concerns the stability of spectral comparisons under graph operations, and the second asks how far hereditary density assumptions can be pushed toward spectral conclusions.

In \cite{ALNS26+}, Ai, Lei, Ning, and Shi posed the following problem about joins with a new vertex.

\begin{problem}\label{Prob:Main3}
Let $G_1,G_2$ be two graphs with the same vertex set and $\lambda(G_1)\geq \lambda(G_2)$.
Determine the pairs $G_1,G_2$ for which
$\lambda(G_1\vee K_1)\ge \lambda(G_2\vee K_1)$.
\end{problem}
As pointed out in \cite{ALNS26+}, the assertion does not hold for arbitrary pairs of graphs.  It is true when $G_1$ is regular, and the following observation gives a slightly more general form with an independent set of any fixed size.

\begin{theorem}
Let $G$ and $H$ be two graphs on the same number $n$ of vertices.
Suppose that $G$ is regular. If
$\lambda(G) \geq \lambda(H)$,
then, for every $s \geq 1$,
$\lambda(G\vee\overline{K}_s ) \geq \lambda(H\vee\overline{K}_s)$.
\end{theorem}

\begin{proof}
Let $G$ be $d$-regular. Then $d = \lambda(G)$.
By the equality case of (2) in Lemma \ref{lem:join-coronal}, we have
$\lambda(\overline{K}_s\vee G) = (d + \sqrt{d^2 + 4ns})/2$.
On the other hand, Lemma \ref{lem:join-coronal} implies that
\[
\lambda(\overline{K}_s\vee H) \leq \frac{\lambda(H) + \sqrt{\lambda(H)^2 + 4ns}}{2}.
\]
Consider the function $f(x) = (x + \sqrt{x^2 + 4ns})/2$, $x \geq 0$.
One can check that this function is increasing on $[0,\infty)$.
Since $\lambda(H) \leq \lambda(G) = d$, it follows that
\[
\frac{\lambda(H) + \sqrt{\lambda(H)^2 + 4ns}}{2}
   \leq \frac{d + \sqrt{d^2 + 4ns}}{2}.
\]
Combining the two inequalities, we obtain $\lambda(\overline{K}_s\vee G) \geq \lambda(\overline{K}_s\vee H)$.
\end{proof}

Thus, the regularity of graphs is enough to make the join operation preserve the spectral-radius comparison.  It would be interesting to identify the weakest structural hypotheses under which this remains true.  In view of Lemma~\ref{lem:join-coronal}, the problem can also be phrased in terms of comparing the coronals $\chi_G(x)$ and $\chi_H(x)$ near the relevant Perron root; this formulation may be useful for irregular graphs.

\end{document}